\documentclass[12pt]{article}
\usepackage{amsmath,amsthm,amssymb,amsfonts,amssymb,latexsym}
\usepackage{fullpage}

\usepackage[utf8]{inputenc}
\usepackage[english]{babel}

\newtheorem{theorem}{Theorem}[section]

\newtheorem{cor}{Corollary}[section]

\newtheorem{defn}{Definition}[section]
\newtheorem{remark}{Remark}[section]
\newtheorem{ex}{Example}[section]

\newcommand{\be}{\begin{equation}}
\newcommand{\ee}{\end{equation}}

\newcommand{\RR}{\mathbb{R}}
\newcommand{\CC}{\mathbb{C}}

\newcommand{\mnr}{\mathbb{R}^{n\times n}}

\renewcommand{\bar}{\overline}

\newfont{\bb}{msbm10}

\renewcommand{\l}{\langle}
\renewcommand{\r}{\rangle}

\newcommand{\NCONF}{\mbox{\bf NCON1}}
\newcommand{\NCONS}{\mbox{\bf NCON2}}

\begin{document}
	\title{How to Detect and Construct N-matrices}
	\author{Projesh Nath Choudhury\thanks{Department of Mathematics,
			Indian Institute of Science, Bengaluru, India
			 (projeshnc@alumni.iitm.ac.in, projeshc@iisc.ac.in).}  		\and
		Michael J. Tsatsomeros\thanks{
		Department of Mathematics and Statistics, 
		Washington State University,
		Pullman, WA 99164, USA
		(tsat@wsu.edu).}}
	\maketitle

\begin{abstract}
N-matrices are real $n\times n$ matrices all of whose principal 
minors are negative. We provide (i) an $O(2^n)$ test to detect whether or 
not a given matrix is an N-matrix, and (ii) a characterization 
of N-matrices, leading to the recursive construction of every N-matrix.
\end{abstract}

{\bf AMS Subject Classification:}
15A48, 
15A15, 
15A23, 
68Q15,  
90C33.  

{\bf Keywords.} N-matrix, P-matrix, almost P-matrix, Principal submatrix, 
                Principal minor, Schur complement.

\newpage
\section{Introduction and Motivation}
\label{intro}
This work concerns N-matrices, that is, real $n\times n$ matrices, $A\in\mnr$, all of 
whose principal minors are negative. 

In prior discussions of N-matrices, their resemblance to P-matrices, 
which are matrices all of whose principal minors are positive, invariably 
comes up first. Indeed, P-matrices are widely studied since they contain 
several prominent classes of matrices, such as the positive 
definite matrices and the M-matrices; they find applications in 
mathematical programming, the study of univalence and complexity theory 
(see e.g., \cite{bepl:94, JST, MP}). 
N-matrices find similar applications and possess properties analogous to P-matrices; 
they were introduced in \cite{I} and have been studied e.g., in \cite{M, PR, ST},
along with special types of N-matrices in \cite{GAJ1, GAJ2}.

Among the motivating factors for studying N-matrices (see also the conclusions
in Section \ref{conclusions}) is their connection 
to univalence (injectivity of differential maps in $\RR^n$) 
and their role in the Linear Complementarity Problem (see Section \ref{back}). 
In addition, as it is evident in the existing theory of N-matrices and will be 
reinforced by the results herein, it is illuminating to identify and compare the 
effects of having signed principal minors in the two cases of N-matrices 
and P-matrices. There are similarities, distinctions, but also some 
unexpected connections between the two classes. Such instances will 
surface in our study of how to 
(i) detect N-matrices efficiently (Section \ref{test}), and 
(ii) construct all the N-matrices (Section \ref{con}).

Some background material and basic properties of N-matrices 
are reviewed in Section \ref{back}, which will help us develop and appropriately 
frame the results. Matlab implementations of algorithms for the detection of 
N-matrices and P-matrices are included in Section \ref{matlab} for the 
reader's convenience.

\section{Background, Notation and Context}
\label{back}

For a positive integer $n$, let $\l n\r=\{1,2,\ldots,n\}$.
For $\alpha\subseteq\l n\r$, $|\alpha|$ denotes the cardinality of $\alpha$
and $\bar\alpha=\l n\r\setminus\alpha$. For $\alpha\subseteq\l n\r$ with
$|\alpha|=k$ and its elements arranged in ascending order, we let $x[\alpha]$ 
denote the vector in $\CC^k$ obtained from the entries of $x\in\CC^n$ 
indexed by $\alpha$. 
Moreover, we let $A[\alpha,\beta]$ denote the submatrix of $A\in\mnr$ whose rows and
columns are indexed by $\alpha,\beta\subseteq \l n\r$, respectively;
the elements of $\alpha,\beta$ are assumed to be in ascending order.
When a row or column index set is empty, the corresponding submatrix
is considered vacuous and by convention has determinant equal to 1.
We abbreviate $A[\alpha,\alpha]$ by $A[\alpha]$ and refer to it as a 
{\em principal submatrix} of $A$ and its determinant as a {\em principal minor} of $A$.

Given $A\in\mnr$ and $\alpha\subseteq\l n\r$ such that $A[\alpha]$ is
invertible, $A/A[\alpha]$ denotes the {\em Schur complement of} $A[\alpha]$
{\em in} $A$, that is, 
\[  A/A[\alpha]\,=\, A[\bar\alpha]\, -\, A[\bar\alpha,\alpha]A[\alpha]^{-1}
A[\alpha,\bar\alpha]. \] 

\begin{defn}
{\rm 
Matrix $A=[a_{ij}]\in\mnr$ is
\begin{itemize}
\item 
an {\em N-matrix} if $\det A[\alpha]<0$ for all $\alpha\subseteq\l n\r$;
\item 
a {\em P-matrix} if $\det A[\alpha]>0$ for all $\alpha\subseteq\l n\r$;
\item 
an {\em almost P-matrix} if $\det A[\alpha]>0$ for all proper $\alpha\subseteq\l n\r$
and $\det A< 0$.
\end{itemize}
We further classify an N-matrix $A\in\mnr$ as being 
\begin{itemize}
\item 
{\em of the first category} if there exist $i,j\in\l n\r$ such that $a_{ij}> 0$; or
\item 
{\em of the second category} if $a_{ij}<0$ for all $i,j\in\l n\r$;\end{itemize}
}
\end{defn}

Array inequalities in the sequel are meant to be entrywise.

For reference and context needed in our further considerations, we gather below some 
analogous properties of N-matrices and P-matrices.  
For the definition and background on the Linear Complementarity problem LCP$(A,q)$,
$A\in\mnr$, $q\in\RR^n$, see \cite{CPS}. For the definition and properties of 
the Principal Pivot Transform ppt$(A, \alpha)$, $A\in\mnr, \alpha\subseteq\l n\r$,
see \cite{T}.

\underline{\bf N-matrices:}
\begin{itemize}
	\item{\bf [N1]} 
	$A\in\mnr$ is an N-matrix if and only if $A^{-1}$ is an almost P-matrix \cite{OPR}.
	\item{\bf [N2]} 
	$A\in\mnr$ is an N-matrix of the second category if and only if for every $q>0$, 
	LCP$(A,q)$ has exactly $2$ solutions \cite{PR}.
	\item{\bf [N3]}
	$A\in\mnr$ is an N-matrix of the second category if and only if
	      $A$ does not reverse the sign of any nonzero, unisigned vector 
	      $x=[x_i]\in\RR^n$;
	      i.e., $(Ax)_i\, x_i \leq 0$ for all $i\in\l n\r$ 
	      implies $x\geq 0$ or $x\leq 0$ \cite{PR}.
	\item{\bf [N4]} 
	If $A\in\mnr$ is an N-matrix, then $A/A[\alpha]$ is a P-matrix
	      for all proper subsets $\alpha$ of $\l n\r$ \cite{ST}.
	\item{\bf [N5]}
	Let $A\in\mnr$ be an N-matrix, $\alpha$ be a proper 
	subset of $\l n\r$ and $B = {\rm ppt}(A,\alpha)$.
	Then $\det(B[\alpha]) < 0$ and all other principal minors of $B$ are positive
	\cite{ST}.
	\item{\bf [N6]} 
	N-matrices have exactly one real negative eigenvalue \cite{PR}.
\end{itemize}

\underline{\bf P-matrices:} \
See \cite[Chapter 3]{JST} or \cite{survey} for a treatment of P-matrices.
\begin{itemize}
	\item{\bf [P1]}
	 $A\in\mnr$ is a P-matrix if and only if $A^{-1}$ is a P-matrix.
	\item{\bf [P2]}
	$A\in\mnr$ is a P-matrix if and only if for every $q\in\RR^n$, 
	LCP$(A,q)$ has a unique solution.
	\item{\bf [P3]}
	$A\in\mnr$ is a P-matrix if and only if
	$A$ does not reverse the sign of any nonzero vector  $x=[x_i]\in\RR^n$;
	i.e., $(Ax)_i\, x_i \leq 0$ for all $i\in\l n\r$ implies $x= 0$.
	\item{\bf [P4]}
	If $A\in\mnr$ is a P-matrix, then $A/A[\alpha]$ is a P-matrix 
	      for all $\alpha\subseteq\l n\r$.
	\item{\bf [P5]}
	$A\in\mnr$ is a P-matrix if and only if ppt$(A,\alpha)$ is a P-matrix
	      for any (and thus all) $\alpha\subseteq\l n\r$.
	\item{\bf [P6]} 
	P-matrices have no real negative eigenvalues.
\end{itemize}

\section{Detecting N-matrices}
\label{test}
The problem of detecting P-matrices is known to be co-NP-complete \cite{Coxson}.
The computation of all the principal minors of $A\in\mnr$ via row reduction leads 
to an $O(n^3\,2^n)$ effort. A more efficient (but still exponential) algorithm 
to compute all principal minors of a square matrix is developed in \cite{GT1}; 
the inverse problem of constructing a matrix from a feasible set of principal minors 
is solved in \cite{GT2}. In the same vein, an efficient, recursive 
algorithm to detect P-matrices of $O(2^n)$ time complexity is 
developed in \cite{TsatLei}. These algorithms are based on the following theorem.
  
\begin{theorem} {\rm \cite[Theorem 3.1]{TsatLei}}
	\label{Ptest}
	Let $A\in\mnr$ and $\alpha \subseteq \l n\r$ with 
	$\lvert \alpha \rvert =1$. Then $A$ is a P-matrix if and only if 
	$A[\alpha],~A[\bar\alpha]$ and $A/A[\alpha]$ are P-matrices.
\end{theorem}

We can extend the theorem above into the following characterization of N-matrices. 

\begin{theorem}\label{Ntest}
	Let $A\in\mnr$ and $\alpha \subseteq \l n\r$ with 
	$\lvert \alpha \rvert =1$. Then $A$ is an N-matrix if and only if 
	$A[\alpha],~A[\bar\alpha]$ are N-matrices and $A/A[\alpha]$ is a P-matrix.
\end{theorem}

\begin{proof}
	Without loss of generality, let $\alpha=\{1\}$; otherwise, our considerations apply to 
	a permutation similarity of $A$. Suppose that $A$ is an N-matrix. 
	By definition, $A[\alpha]$ and $A[\bar\alpha]$ are N-matrices. 

By {\bf [N4]}, $A/A[\alpha]$ is a P-matrix.
		
Conversely, suppose $A[\alpha]$ and $A[\bar\alpha]$ are N-matrices and $A/A[\alpha]$ is 
a P-matrix. The determinant of any principal submatrix of $A$ without any entries 
from the first column is a principal minor of $A[\bar\alpha]$ and it is thus negative. 
Let $B$ be any principal submatrix of $A$ with entries from the first column of $A$. 
Then $C=A[\alpha]$ is a principal submatrix (diagonal entry) of $B$, 
and $B/C$ is a principal submatrix of $A/A[\alpha]$. 
Thus $\det(B/C)>0$ and so $\det(B)=A[\alpha] \det(B/C)<0$. 
Hence $A$ is an N-matrix.
\end{proof}

Theorem \ref{Ntest} suggests the following recursive algorithm 
for detecting N-matrices.

	ALGORITHM $N(A)$
	\begin{enumerate}
		\item Input $A=[a_{ij}]\in\mnr$
		\item If $a_{11}\geq 0$, output ``A is not an N-matrix'' stop
		\item Compute $A/a_{11}$
		\item If $A/a_{11}$ is not P-matrix output ``A is not an N-matrix'' stop
		\item Call $N(A[\bar{\{1\}}])$
		\item Output ``$A$ is an N-matrix''
	\end{enumerate}

A Matlab implementation of algorithm $N(A)$ is found in Section \ref{matlab} (NTEST). 
The algorithm needed in (step 4) of NTEST to detect a P-matrix is based on Theorem 
\ref{Ptest} and also provided in Section \ref{matlab} (PTEST).
   

\section{Constructing All N-matrices}
\label{con}
Examples of N-matrices, even of special structure and form, are not as
easy to generate as is for examples of P-matrices.  Some possibilities include 
the types of N-matrices considered in \cite{Fan} and \cite{GAJ1, GAJ2},
as well as the totally negative matrices (all minors are negative) in \cite{CRU,FV}.
In \cite[Theorem 7.12]{CKS}, some necessary conditions are presented on the 
signs of the entries of an N-matrix of the first category.

In this section, using a recursion based on rank-one perturbations of N-matrices, 
we can reverse the steps of the recursive algorithm $N(A)$ that detects N-matrices 
and thus construct {\bf every} N-matrix of either category. This approach is based
on the following corollary of Theorem \ref{Ntest}.

\begin{cor}\label{Const_N_2}
Let $A\in\mnr$ be an N-matrix of the second category, $a\in\RR$ and 
let $x,y \in \RR^n$. Then the following are equivalent:
	\begin{itemize}
		\item[\rm (i)] $U=\begin{bmatrix}
		A & x\\y^T & a
		\end{bmatrix}$
		is an N-matrix of the second category.
		\item[\rm (ii)] $a,x,y<0$ and $\displaystyle{A-\frac{1}{a}xy^T}$ is a P-matrix.
	\end{itemize}
\end{cor}

Corollary \ref{Const_N_2} allows us to recursively construct 
$n \times n~ (n\geq 2)$ N-matrices of the second category as follows.

ALGORITHM $\NCONS$

\begin{enumerate}
	\item Choose $A_1<0$
	\item For $i=1:{n-1}$, given the $i \times i\;$ N-matrix of the 
	second category $A_i$,
	\begin{itemize}
		\item[(a)] choose $a_{i}<0$ and $x^{(i)}, y^{(i)}\in \mathbb{R}^i$ 
		such that $x^{(i)}, y^{(i)}<0$ and
		$A_i - \frac{1}{a_i}x^{(i)} {y^{(i)}}^T$ is a P-matrix
		\item[(b)] construct the ${(i+1)} \times {(i+1)}$ matrix 
		$A_{i+1}=\begin{bmatrix}
		A_i & x^{(i)}\\ {y^{(i)}}^T & a_i
		\end{bmatrix}$
	\end{itemize}
	\item Output ``$A=A_n$ is an N-matrix of the second category''
\end{enumerate}

\begin{theorem}
	Every matrix constructed by $\NCONS$ is an N-matrix of the second category. 
	Conversely, every N-matrix of the second category can be constructed by $\NCONS$.
\end{theorem}

\begin{proof}
	By Corollary \ref{Const_N_2}, the sequence of matrices $A_{i+1}~(i=1,\ldots, {n-1})$ 
	constructed by $\NCONS$, including $A_{1}$, are N-matrices of the second category. To 
	prove the converse, we proceed by induction on the order of matrices. The statement is 
	trivial for $n=1$. Let $n\geq 2$ and suppose that every N-matrix of the second 
	category of order smaller than $n$ can be constructed by $\NCONS$. 
	Let $A \in \mnr$ be an N-matrix of the second category. 
	Then $A$ can be partitioned as 
	\begin{center}
		$A=\begin{bmatrix}
		A_{n-1} & u\\v^T & a
		\end{bmatrix},$
	\end{center}
where $A_{n-1} \in \mathbb{R}^{(n-1)\times (n-1)}$ is N-matrix of the second category, 
$u,v \in \mathbb{R}^{n-1}$ and $a\in \mathbb{R}$. By inductive hypothesis, $A_{n-1}$ can 
be constructed by $\NCONS$. Since $A$ is N-matrix of the second category, by Corollary 
\ref{Const_N_2}, $A_{n-1}/a=A_{n-1}-\frac{1}{a}uv^T$ is a P-matrix and $a,x,y<0$. Thus 
$A_n=A$ can be constructed by $\NCONS$ with the following choices:
\begin{center}
	$a_{n-1}=a,~x^{(n-1)}=u \hbox{ and }y^{(n-1)}=v$.
\end{center}
\end{proof}

To extend our construction methodology to N-matrices of the first category,
we recall the following result.

\begin{theorem} {\rm \cite[Theorem 7.9.4]{baprag}}
	\label{persimilarity}
Let $A$ be a N-matrix of the first category. Then there exists a 
permutation matrix $P$ such that 
\begin{equation}\label{1steqn}
	PAP^T=\begin{bmatrix}
	A_{11} & A_{12}\\A_{21} & A_{22}
	\end{bmatrix},
	\end{equation}
where $A_{11},~A_{22}<0$ are square matrices and $A_{12}, A_{21}>0$.
\end{theorem}

By Theorem \ref{persimilarity}, in order to construct all N-matrices 
of the first category of size $n\geq 2$, it is sufficient to construct them in the 
form \eqref{1steqn}, where $A_{11}\in \mathbb{R}^{k \times k}$  $(k<n)$.
This can be achieved using the following Corollary of Theorem \ref{Ntest}.
\begin{cor}\label{Const_N_1}
	Let $A=\begin{bmatrix}
	A_{11} & A_{12} \\ A_{21} & A_{22}
	\end{bmatrix}\in\mnr ~(n\geq 2)$ be an N-matrix, where $A_{11}\in \mathbb{R}^{k\times k} ~(k\leq n)$, $A_{22}\in \mathbb{R}^{{n-k}\times {n-k}}$ with $A_{11}, A_{22}<0$ and $A_{12},~A_{21}>0$. Let $a\in\RR$ and 
	let $x,y \in \RR^n$ . Then the following are equivalent:
	\begin{itemize}
		\item[\rm (i)] $U=\begin{bmatrix}
		A & x\\y^T & a
		\end{bmatrix}$
		is an N-matrix of the first category.
		\item[\rm (ii)] $a<0$, $A-\frac{1}{a}xy^T$ is a P-matrix,
		$x[\l k\r], y[\l k\r]>0$ and 
		$x[\bar{\l k\r}], y[\bar{\l k\r}]<0$.
	\end{itemize}
\end{cor}

Using Corollary \ref{Const_N_1}, we can construct N-matrices 
of the first category as follows.

ALGORITHM $\NCONF$
\begin{enumerate}
	\item Construct $A_k=A_{11}$ using algorithm $\NCONS$
	\item For $i=k:{n-1}$, given the $i \times i$ matrix $A_i$,
	\begin{itemize}
		\item[(a)] choose $a_i<0$, $x^{(i)}, y^{(i)}\in \mathbb{R}^i$ 
		such that $x[\l k\r], y[\l k\r]>0,$ 
		$x[\bar{\l k\r}], y[\bar{\l k\r}]<0$, and 
		$A_i - \frac{1}{a_i}x^{(i)} {y^{(i)}}^T$ is a P-matrix
		\item[(b)] construct the ${(i+1)} \times {(i+1)}$ matrix 
		$A_{i+1}=\begin{bmatrix}
		A_i & x^{(i)}\\ {y^{(i)}}^T & a_i
		\end{bmatrix}$
	\end{itemize}
	\item Output ``$A=A_n$ is an N-matrix of the first category''
\end{enumerate}

\begin{theorem}
	Every matrix constructed by $\NCONF$ is an N-matrix of the first category. 
	Conversely, every N-matrix of the first category can be constructed as a permutational similarity of a matrix constructed by $\NCONF$.
\end{theorem}

\begin{proof}
	By Corollary \ref{Const_N_1}, the sequence of matrices $A_{i+1}~(i=1,\ldots, {n-1})$ 
	constructed by $\NCONF$, are N-matrices of the first category. We use induction on the 
	order of matrices to prove the converse. The base case $n=2$ is obvious. Let $n>2$ and 
	suppose that every N-matrix of the first category of order smaller than $n$ can be 
	constructed as a permutational similarity of a matrix constructed by $\NCONF$. Let 
	$A \in \mnr$ be an N-matrix of the first category. Then there exist a permutation 
	matrix $P$ such that
	\begin{center}
		$PAP^T=\begin{bmatrix}
		A_{11} & A_{12}\\A_{21} & A_{22}
		\end{bmatrix},$
	\end{center}
where $A_{11}\in \mathbb{R}^{k\times k} ~(k<n)$, $A_{22}\in \mathbb{R}^{{n-k}\times {n-k}}$ with $A_{11}, A_{22}<0$ and $A_{12},~A_{21}>0$. Let $\begin{bmatrix}
A_{11} & A_{12}\\A_{21} & A_{22}
\end{bmatrix}=\begin{bmatrix}
A_{n-1} & u\\v^T & a
\end{bmatrix}$, where $A_{n-1} \in \mathbb{R}^{{n-1}\times {n-1}}$ is N-matrix, $a<0$, 
$u,v \in \mathbb{R}^{n-1}$ with $u[\l k\r], v[\l k\r]>0$ and 
$u[\bar{\l k\r}], v[\bar{\l k\r}]<0$. Now, either $A_{n-1}<0$ or $A_{n-1}$ is of the form \eqref{1steqn}. 
By inductive hypothesis, $A_{n-1}$ can be constructed using $\NCONF$. 
Since $A$ is N-matrix of the first category, by Corollary \ref{Const_N_1}, 
$A_{n-1}/a=A_{n-1}-\frac{1}{a}uv^T$ is a P-matrix. Thus $A_n=P^TAP$ can be 
constructed by $\NCONF$ with the following choices:
	\begin{center}
		$a_{n-1}=a,~x^{(n-1)}=u \hbox{ and }y^{(n-1)}=v$.
	\end{center}
\end{proof}

\begin{remark} $\!$
	
	{\em (1) The implementation of step 2(a) in algorithms
		$\NCONF$ and $\NCONS$ can be done via random choice of the 
		appropriately signed vectors $x^{(i)}$ and $y^{(i)}$ and judicious 
		choice of the diagonal entries $a_i$. The process of choosing $a_i$ so that 
		$\displaystyle{A_i - \frac{1}{a_i}x^{(i)} {y^{(i)}}^T}$ 
		is a P-matrix is developed and its effects explained
		in the recursive construction of all P-matrices presented in
		\cite[Section 4]{tz}.
	
	    (2) In light of [N1] in Section \ref{back}, Algorithms $\NCONF$ and $\NCONS$ may 
	    also be viewed as methods to construct almost P-matrices via inversion. 
             }
\end{remark}

We proceed with two illustrative examples of N-matrices constructed using
$\NCONF$ and $\NCONS$. 

\begin{ex}
	{\rm 
		We construct $3 \times 3$ N-matrix of the first category.
		Let $A_1=[-1]$, $a_1=-1,~x^{(1)}=[2]$ and $y^{(1)}=[2]$. 
		Then $A_1-\frac{1}{a_1}x^{(1)} {y^{(1)}}^T=[3]$ is a P-matrix. 
		By $\NCONF$, $A_2=\begin{bmatrix}
		-1 & 2\\2 &-1
		\end{bmatrix}$ is N-matrix of the first category. Now, let 
		$a_2=-1, ~x^{(2)}=[2\;\; -1]^T$ and $y^{(2)}=[2 \;\; -2]^T$. 
		Then $A_2-\frac{1}{a_2}x^{(2)} {y^{(2)}}^T=\begin{bmatrix}
		3 & -2\\ 0 & 1
		\end{bmatrix}$ is a P-matrix. Again, by $\NCONF$, $A_3=\begin{bmatrix}
		-1 & 2 & 2\\2 &-1 & -1 \\2 & -2 & -1
		\end{bmatrix}$ is N-matrix of the first category.
	}
\end{ex}

\begin{ex}
{\rm
In this example, we construct a $3 \times 3$ N-matrix of the second category 
by $\NCONS$.
Let $A_1=[-1]$, $a_1=-1,~x^{(1)}=[-1]$ and $y^{(1)}=[-2]$. 
Then $A_1-\frac{1}{a_1}x^{(1)} {y^{(1)}}^T=[1]$ is a P-matrix. 
By $\NCONS$, $A_2=\begin{bmatrix}
	-1 & -1\\-2 &-1
	\end{bmatrix}$ is N-matrix of the second category. 
	Now, we take $a_2=-1,~x^{(2)}=[-2\;\; -1]^T$ and 
	$y^{(2)}=[-3\;\; -2]^T$. 
	Then $A_2-\frac{1}{a_2}x^{(2)} {y^{(2)}}^T=\begin{bmatrix}
	5 & 3\\ 1 & 1
	\end{bmatrix}$ is a P-matrix. Hence, by $\NCONS$, $A_3=\begin{bmatrix}
	-1 & -1 & -2\\-2 &-1 & -1 \\-3 & -2 & -1
	\end{bmatrix}$ is N-matrix of the second category.
} 
\end{ex}
 
\section{NTEST and PTEST}
\label{matlab}

We include Matlab code for the detection of P-matrices and N-matrices.

{\bf PTEST} (detects P-matrices)\\
{\tt \noindent
	function [r] = ptest(A)\\
	\% Return r=1 if `A' is a P-matrix (r=0 otherwise). 
	
	n = length(A);\\
	if $\sim$(A(1,1)>0), r = 0;\\
	elseif n==1, r = 1;   \\
	else\\
	\hspace*{.1in}   b = A(2:n,2:n);\\
	\hspace*{.1in}   d = A(2:n,1)/A(1,1);\\
	\hspace*{.1in}   c = b - d*A(1,2:n);  \\
	\hspace*{.1in}   r = ptest(b) \& ptest(c);\\
	end
}

\medskip
{\bf NTEST} (detects N-matrices)\\
{\tt \noindent
	function [r] = ntest(A)\\
	\% Return r=1 if `A' is a N-matrix (r=0 otherwise). 
	
	n = length(A);\\
	if $\sim$(A(1,1)<0), r = 0;\\
	elseif n==1, r = 1;   \\
	else\\
	\hspace*{.1in}   b = A(2:n,2:n);\\
	\hspace*{.1in}   d = A(2:n,1)/A(1,1);\\
	\hspace*{.1in}   c = b - d*A(1,2:n);  \\
	\hspace*{.1in}   r = ntest(b) \& ptest(c);\\
	end
}

Note that the time complexity of PTEST is $O(2^n)$ \cite{TsatLei}, and 
so this must also be the case for NTEST as the same binary tree of 
matrices (of orders recursively reduced by $1$) is being
created by the two algorithms. 

\section{Conclusions}
\label{conclusions}

N-matrices are challenging to detect and understand their nature. The progress reported
allows for their further consideration in applications. One can now detect 
such matrices by an algorithm that because of its recursive nature can be implemented in
parallel. One can also construct generic N-matrices (of either category) for purposes 
of experimentation and development of new theory and algorithms. Moreover, the work 
herein contributes in better understanding the role of the signs of the 
principal minors in the theory of inequalities and in the study of computational complexity (generally and within the confines of complementarity problems).   


\end{document}